\documentclass[12pt]{amsart}
\usepackage{amsmath, amsthm, amscd, amsfonts, amssymb, graphicx, color}
\usepackage[bookmarksnumbered, colorlinks, plainpages]{hyperref}
\hypersetup{colorlinks=true,linkcolor=red, anchorcolor=green, citecolor=cyan, urlcolor=red, filecolor=magenta, pdftoolbar=true}
\input{mathrsfs.sty}

\newtheorem{theorem}{Theorem}[section]
\newtheorem{lemma}[theorem]{Lemma}
\newtheorem{proposition}[theorem]{Proposition}
\newtheorem{corollary}[theorem]{Corollary}
\theoremstyle{definition}
\newtheorem{definition}[theorem]{Definition}

\theoremstyle{remark}
\newtheorem{remark}[theorem]{Remark}
\numberwithin{equation}{section}

\begin{document}

\title[Non-commutative Stein inequality]{Non-commutative Stein inequality and its applications}

\author[A. Talebi, M.S. Moslehian ]{Ali Talebi and Mohammad Sal Moslehian }

\address{ Department of Pure Mathematics, Ferdowsi University of
Mashhad, P.O. Box 1159, Mashhad 91775, Iran.}
\email{alitalebimath@yahoo.com}
\email{moslehian@um.ac.ir and moslehian@member.ams.org}



\begin{abstract}
The non-commutative Stein inequality asks whether there exists a constant $C_{p,q}$ depending only on $p, q$ such that
\begin{equation*}
\left\| \left(\sum_{n} |\mathcal{E}_{n} (x_n) |^{q}\right)^{\frac{1}{q}} \right\|_p \leq
C_{p,q} \left\| \left(\sum_{n} | x_n |^q \right)^{\frac{1}{q}}\right \|_p\qquad \qquad (S_{p,q}),
\end{equation*}
for (positive) sequences $(x_n)$ in $L_p(\mathcal{M})$. The validity of $(S_{p,2})$ for $1 < p < \infty$ and $(S_{p,1})$ for $1 \leq p < \infty$ are known.
In this paper, we verify (i)  $(S_{p,\infty})$ for $1 < p \leq \infty$; (ii) $(S_{p,p})$ for $1 \leq p < \infty$; (iii) $(S_{p,q})$ for $1 \leq q \leq 2$ and $q<p<\infty$. We also present some applications.
\end{abstract}
\maketitle
%
%
\section{Introduction}

Throughout this note, $(\mathcal{M}, \tau)$ denotes a non-commutative probability space, that is, a von Neumann algebra $\mathcal{M}$ equipped with a normal faithful finite trace $\tau$
with $\tau(1)=1$, where $1$ stands for the identity of $\mathcal{M}$.

For $1 \leq p < \infty$ the space $L_p(\mathcal{M})$ is the completion of $\mathcal{M}$ with respect to the $p$-norm $\|x\|_p:=\tau(|x|^p)^{1/p}$, where $|x|=(x^*x)^{1/2}$ is
the absolute value of $x$. If $p=\infty$, $L_p(\mathcal{M})$ is $\mathcal{M}$ itself with the operator norm. We denote by $L_p^+(\mathcal{M})$ the positive cone in $L_p(\mathcal{M})$ consisting of those elements, which are limits of sequences of positive elements in $\mathcal{M}$.

Assume that $\mathcal{N}$ is a von Neumann subalgebra of $\mathcal{M}$. Then there exists a map $\mathcal{E}_{\mathcal{N}}: \mathcal{M} \longrightarrow \mathcal{N}$, named the conditional
expectation, which satisfies the following properties:
\begin{itemize}
\item[(i)] $\mathcal{E}_{\mathcal{N}}$ is normal positive contractive projection from $\mathcal{M}$ onto $\mathcal{N}$;
\item[(ii)] $\mathcal{E}_{\mathcal{N}}(axb) = a \mathcal{E}_{\mathcal{N}}(x) b$ for every $x \in \mathcal{M}$ and $a, b \in \mathcal{N}$;
\item[(iii)] $\tau \circ \mathcal{E}_{\mathcal{N}} = \tau$.
\end{itemize}
Moreover, $\mathcal{E}_{\mathcal{N}}$ is the unique map verifying {\rm (ii), (iii)}.
It is known that $\mathcal{E}_{\mathcal{N}}$ can be extended to a contraction, denoted by the same $\mathcal{E}_{\mathcal{N}}$, from $L_p(\mathcal{M})$ into $L_p(\mathcal{N})$.

By a filtration we mean an increasing sequence  $(\mathcal{M}_n)_{n\ge 0}$ of von Neumann subalgebras of $\mathcal{M}$ such that $\bigcup\limits_{n\ge 0} \mathcal{M}_n$
generates $\mathcal{M}$ in the $w^*$-topology. A sequence $(x_n)_{n \geq 0}$ in $L_p(\mathcal{M})$ is said to be adapted to $(\mathcal{M}_n)_{n\ge 0}$ if $x_n \in L_p \left( \mathcal{M}_n \right)$ ~ ($n \geq 0$). The reader is referred to \cite{sm1, TMS, TMS2, Xu} for more information on non-commutative probability spaces.

Let $(\Omega, \mathcal{F}, \mathbb{P})$ be a probability space, $(\mathcal{F}_{n})_{n=0}^\infty$ be a filtration on it and $(X_n)_{n=1}^\infty$ be a stochastic process adapted to
$(\mathcal{F}_n)_{n=0}^\infty$. The so-called Stein inequality
\begin{eqnarray}\label{equ1}
\left\| (\sum_{n=1}^\infty |\mathbb{E}_{\mathcal{F}_{n-1}}X_n|^q)^{\frac{1}{q}} \right\|_p
\leq C_{p,q}\left\| (\sum_{n=1}^\infty |X_n|^q)^{\frac{1}{q}}\right\|_p\qquad \qquad (S_{p,q})
\end{eqnarray}
was proved by Stein \cite{S} for $1 < p < \infty$ and $q = 2$ and by Asmar and Montgomery-Smith \cite{AMS} for any $1 < p < \infty$ and $1 \leq q \leq \infty$.
Using the duality between the martingale Hardy spaces and $\mbox{BMO}$-spaces, Lepingle \cite{Lepingle} verified inequality \eqref{equ1} for adapted process with $q=2$, $p=1$ and $C_{p,q}=2$. In a paper of Bourgain (\cite{BOU}, proposition 5) inequality \eqref{equ1} was obtained with a constant $3$. It is evident that non-commutative probability theory is inspired by classical probability theory and quantum mechanics. Several mathematicians investigated and applied inequality
\eqref{equ1} in the non-commutative setting.

In 1997, Pisier and Xu \cite{PX} proved a non-commutative Stein inequality for $1 < p < \infty$ and $q =2$. They indeed proved that there is
a constant $C_p$ depending on $p$ such that for any finite sequence $(x_n)_{n=1}^N$ in $L_p(\mathcal{M})$,
\begin{equation*}
\left\| \left(\sum_{n=1}^N |\mathcal{E}_{n} (x_n) |^{2}\right)^{\frac{1}{2}} \right\|_p \leq
C_p \left\| \left(\sum_{n=1}^N | x_n |^2 \right)^{\frac{1}{2}}\right \|_p,
\end{equation*}
See also \cite{RAN} for another proof. In 2002, Junge \cite{J} verified the non-commutative analogue of \eqref{equ1} for $1 \leq p < \infty$ and $q=1$.

In 2014, Qiu \cite{Qiu} employ a duality argument mentioned to show that if
$\left( \mathcal{M}_n \right)_{n=0}^N$ is an increasing filtration of von Neumann subalgebras and $(x_n)_{n=1}^N$ is an adapted sequence in $L_1(\mathcal{M})$, then
\begin{equation*}
\left\| \left(\sum_{n=1}^N |\mathcal{E}_{n-1} (x_n) |^{2}\right)^{\frac{1}{2}} \right\|_1 \leq
2 \left\| \left(\sum_{n=1}^N | x_n |^2 \right)^{\frac{1}{2}}\right \|_1,
\end{equation*}
where $\mathcal{E}_{n}$ denotes the conditional expectation with respect to $\mathcal{M}_n$.

The aim of this note is to obtain some versions of the non-commutative Stein inequality \eqref{equ1} for other values $p$ and $q$. More precisely, we verify (i)  $(S_{p,\infty})$ for $1 < p \leq \infty$ and $q=\infty$ in Theorem \ref{th2}; (ii) $(S_{p,p})$ for $1 \leq p < \infty$ in Proposition \ref{msmz1}; (iii) $(S_{p,q})$ for $1 \leq q \leq 2$ and $q< p < \infty$ in Corollary \ref{msmz2}.

\section{Stein inequality in noncommutative settings}

Suppose that $\left( X, \mathcal{F}, \mu \right)$ is a measure space. The Banach space of all sequences $f = \left( f_n \right)_{n \geq 1}$ of functions such that the norm
\begin{align*}
\| f \|_{p, q} := \left( \int_X \left( \sum_{n} | f_n (x) |^q \right)^{\frac{p}{q}} d\mu (x) \right)^{\frac{1}{p}}
\end{align*}
is finite, is denoted by $L_p(\ell_q)$, in the case that $1 \leq p, q < \infty$.
In the case that $q= \infty$, we set
\begin{align*}
\| f \|_{p, \infty} := \left( \int_X \left( \sup_n | f_n(x) |\right)^p d\mu (x) \right)^{\frac{1}{p}}
\end{align*}
If $p = \infty$, we adopt the natural definition of essential supremum norm.

The behavior of $L_p(\ell_q)$ is very similar to $L_p(X, \mu)$. For instance, the dual space of $L_p(\ell_q)$ ~ ($1 \leq p, q < \infty$) is $L_{p^\prime}(\ell_{q^\prime})$,
where $\frac{1}{p} + \frac{1}{p^\prime} = \frac{1}{q} + \frac{1}{q^\prime} =1$; cf. \cite{S}.
\begin{remark}\label{rem1}
Difficulty of having a non-commutative analogue of the space $L_p(\ell_\infty)$ is the lack of a non-commutative analogue of maximum. It is evident that there may be no maximum of two
positive matrices. So we lead to the notion of $L_p (\mathcal{M}, \ell_\infty )$; cf. \cite{J}.
\end{remark}
We recall that the non-commutative spaces $L_p (\mathcal{M}, \ell_1 )$ and $L_p (\mathcal{M}, \ell_\infty )$ playing an essential role in non-commutative analysis. The results presented
in this section come from \cite{BCO, J, Xu}.
\begin{definition}
The space of all sequences $x = (x_n)_{n \in \mathbb{N}}$ in $L_p(\mathcal{M})$, which can be decomposed as $x_n = a y_n b$ for each $ n \in \mathbb{N}$, where $a, b \in L_{2p} (\mathcal{M})$
and $y = \left( y_n \right)_{n \in \mathbb{N}}$ is in $\mathcal{M}$, is denoted by $L_p (\mathcal{M}, \ell_\infty )$.
It is known that $L_p (\mathcal{M}, \ell_\infty )$ is a Banach space equipped with the norm
\begin{align*}
\| x \|_{L_p (\mathcal{M}, \ell_\infty)} := \inf \left\{ \| a \|_{2p} \, \sup_{n\geq 1} \| y_n \|_\infty \, \| b \|_{2p} \right\},
\end{align*}
where the infimum runs over all possible decompositions of $x$ as above.
\end{definition}
The element $\| x \|_{L_p (\mathcal{M}, \ell_\infty)}$ is denoted by $\| \sup^+_{n} x_n \|_p$. The reader should notice that $\| \sup^+_{n} x_n \|_p$ is a notation for $\sup_{n} x_n$.
Thus we may consider the space $L_p (\mathcal{M}, \ell_\infty)$ as a non-commutative analogue of the classical space $L_p(\ell_\infty)$.

$\bullet$  Let $L_p\left( \mathcal{M}, \ell_1 \right)$ be the space of all sequences $x = (x_n)_{n \in \mathbb{N}}$ in $L_p (\mathcal{M})$, which admits a factorization of the
form $x_n = \sum_{n \in \mathbb{N}} c_{kn}^* d_{kn}$ ~ $(n \in \mathbb{N})$, where $(c_{kn})_{k, n \geq 1}$ and $(d_{kn})_{k, n \geq 1}$ are sequences in $L_{2p} (\mathcal{M})$
such that $\sum_{k, n} c_{kn}^* c_{kn}$ and $\sum_{k,n} d_{kn}^* d_{kn}$ are in $L_p(\mathcal{M})$.

It is known that $L_p\left( \mathcal{M}, \ell_1 \right)$ equipped with the norm
\begin{align*}
\| x \|_{L_p \left( \mathcal{M}, \ell_1 \right)} := \inf \left\{ \left\| \sum_{k, n} c_{kn}^* c_{kn} \right\|_p ^{\frac{1}{2}} \left\| \sum_{k,n} d_{kn}^* d_{kn} \right\|_p^{\frac{1}{2}} \right\},
\end{align*}
where the infimum runs over all possible decompositions of $x$ as above, is a Banach space.

It is proved in \cite{Xu} that the space $L_p\left( \mathcal{M}, \ell_1 \right)$ ~ $(1 \leq p < \infty)$ is the predual of $L_{p^\prime} (\mathcal{M}, \ell_\infty)$ (where $p^\prime$ is
the conjugate to $p$) under the duality
$\langle x, y \rangle = \sum _{n \in \mathbb{N}} \tau \left( x_n y_n \right)$
for $x \in L_p\left( \mathcal{M}, \ell_1 \right)$ and $y \in L_{p^\prime} (\mathcal{M}, \ell_\infty)$.
\begin{remark}
For any positive sequence $x=(x_n)$ in $L_p\left( \mathcal{M}, \ell_1 \right)$ (i.e., $x_n \geq 0$ for all $n$) we have
$\| x \|_{L_p \left( \mathcal{M}, \ell_1 \right)} = \| \sum_{n \in \mathbb{N}} x_n \|_p$.
Therefore one may consider the space $L_p\left( \mathcal{M}, \ell_1 \right)$ as a generalization of the space $L_p(\ell_1)$ in the non-commutative setting.
\end{remark}
In the sequel, we present some useful properties of the spaces $L_p\left( \mathcal{M}, \ell_1 \right)$ and $L_p (\mathcal{M}, \ell_\infty)$.
\begin{theorem}(\cite[Proposition 2.12]{Xu})\label{th1}
Let $1 \leq p \leq \infty$.
\begin{itemize}
\item [(1)] \label{fact1}
Each element in the unit ball of $L_p (\mathcal{M}, \ell_1)$ (resp. $L_p (\mathcal{M}, \ell_\infty)$) is a sum of eight (resp. sixteen) positive elements in the same ball.
\item [(2)]
For any $x \in L_p (\mathcal{M}, \ell_\infty)$ it holds that
\begin{align*}
\| \sup ^+_{n} x_n \|_p = \sup \left\{ \sum_{n \in \mathbb{N}} \tau \left( x_n y_n \right) ~ : y \in L_{p^\prime}\left( \mathcal{M}, \ell_1 \right) ~ \text{and} ~ \| y \|_{L_{p^\prime}\left( \mathcal{M}, \ell_1 \right)} \leq 1 \right\}.
\end{align*}
Moreover, if $x$ is positive,
\begin{align*}
\| \sup ^+_{n} x_n \|_p = \sup \left\{ \sum_{n \in \mathbb{N}} \tau \left( x_n y_n \right) ~ : y_n \in L_{p^\prime}^+ ~ \text{and} ~ \| \sum_{n \in \mathbb{N}} y_n \|_{p^\prime} \leq 1 \right\}.
\end{align*}
\end{itemize}

\end{theorem}
 Some important spaces of sequences in $L_p(\mathcal{M})$ can be formed via the row and column spaces, which are related to Burkholder--Gundy non-commutative inequalities.
\begin{definition}
Let $1 \leq p \leq \infty$ and $x = (x_n)_{n=1}^N$ be a finite sequence in $L_p(\mathcal{M})$. Set two norms
\begin{align*}
\| x \|_{L_p \left(\mathcal{M}; \ell_2^C \right)} := \left\| \left( \sum_{n=1}^N | x_n |^2 \right)^{\frac{1}{2}} \right\|_p ~
\text{and} ~
\| x \|_{L_p \left(\mathcal{M}; \ell_2^R \right)} := \left\| \left( \sum_{n =1}^N | x_n^* |^2 \right)^{\frac{1}{2}} \right\|_p.
\end{align*}
The corresponding completion spaces are denoted by $L_p \left(\mathcal{M}; \ell_2^C \right)$ and $L_p \left(\mathcal{M}; \ell_2^R \right)$, respectively.
\end{definition}
\begin{remark}\label{re1}
Note that $L_p \left(\mathcal{M}; \ell_2^C \right)$ (resp. $L_p \left(\mathcal{M}; \ell_2^R \right)$) is isometric to the column (resp. row) subspace
of $L_p \left(\mathcal{M} \otimes B(\ell_2) \right)$ via the following maps, where $B(\ell_2)$ is the algebra of bounded linear maps on $\ell_2$ with its usual trace $Tr$
and $\mathcal{M} \otimes B(\ell_2)$ denotes the von Neumann tensor product equipped with the (semifinite faithful) tensor trace $\tau \otimes Tr$:
\begin{align*}
x = (x_n)_{n \geq 0} \mapsto
\left(
\begin{array}{ccc}
x_0 & 0 & \ldots \\
x_1 & 0 & \ldots \\
\vdots & \vdots
\end{array} \right)
(\text{resp.} ~ x = (x_n)_{n \geq 0} \mapsto
\left(
\begin{array}{ccc}
x_0 & x_1 & \ldots \\
0 & 0 & \ldots \\
\vdots & \vdots
\end{array} \right) ).
\end{align*}
It is known (cf. \cite[p. 670]{PX}) that
$L_p \left(\mathcal{M}; \ell_2^C \right)^* = L_q \left(\mathcal{M}; \ell_2^C \right) ~ \text{and} ~ L_p \left(\mathcal{M}; \ell_2^R \right)^* = L_q \left(\mathcal{M}; \ell_2^R \right)$,
whenever $\frac{1}{p} + \frac{1}{q} = 1\,\,(1\leq p<\infty)$.
\end{remark}
For notational convenience, we recall the space $CR_p\left[ L_p(\mathcal{M}) \right]$.\\
$\bullet$
For $1 \leq p \leq \infty$, let us define the space $CR_p\left[ L_p(\mathcal{M}) \right]$ in two cases as follows:
\begin{itemize}
\item
If $p \geq 2$,
\begin{equation*}
CR_p\left[ L_p(\mathcal{M}) \right] = L_p \left(\mathcal{M}; \ell_2^C \right) \cap L_p \left(\mathcal{M}; \ell_2^R \right)
\end{equation*}
with the following norm:
\begin{equation*}
\| (x_n) \|_{CR_p\left[ L_p(\mathcal{M}) \right]} = \max \left\{ \| (x_n) \|_{L_p \left(\mathcal{M}; \ell_2^C \right)} , \| (x_n) \|_{L_p \left(\mathcal{M}; \ell_2^R \right)} \right\}.
\end{equation*}
\item
If $p < 2$,
\begin{equation*}
CR_p\left[ L_p(\mathcal{M}) \right] = L_p \left(\mathcal{M}; \ell_2^C \right) + L_p \left(\mathcal{M}; \ell_2^R \right)
\end{equation*}
equipped with the following norm:
\begin{equation*}
\| (x_n) \|_{CR_p\left[ L_p(\mathcal{M}) \right]} = \inf \left\{ \| (a_n) \|_{L_p \left(\mathcal{M}; \ell_2^C \right)} + \| (b_n) \|_{L_p \left(\mathcal{M}; \ell_2^R \right)} \right\},
\end{equation*}
where the infimum runs over all possible decompositions $x_n = a_n + b_n$ with $a_n$ and $b_n$ in $L_p(\mathcal{M})$.
\end{itemize}

In \cite{JX3} Junge and Xu defined an interesting complex interpolation space between $L_p (\mathcal{M}, \ell_\infty)$ and $L_p (\mathcal{M}, \ell_1)$ as
$
L_p (\mathcal{M}, \ell_q) := \left[ L_p (\mathcal{M}, \ell_\infty), L_p (\mathcal{M}, \ell_1) \right]_{\frac{1}{q}}.
$
The reader is referred to \cite{B} for interpolation theory and to \cite{JX3} for some useful properties of the space $L_p (\mathcal{M}, \ell_q)$. Note that if $\mathcal{M}$ is injective,
this definition is a special case of Pisier's vector-valued non-commutative $L_p$-space theory \cite{Pi}.

In the sequel, let $(\mathcal{M}_n)_{n=0}^\infty$ be an increasing sequence of von Neumann subalgebras of $\mathcal{M}$ and $\mathcal{E}_n$ denote the conditional expectation of
$\mathcal{M}$ with respect to $\mathcal{M}_n$. We consider its extension, denoted by the same $\mathcal{E}_n$ from $L_p(\mathcal{M}) $ to $L_p(\mathcal{M}_n)$.

In \cite{J}, a noncommutative Doob's inequality is obtained by the following dual version of Doob's inequality:
\begin{lemma}\label{lem1}
Let $1 \leq p < \infty$. Then for every sequence of positive elements $(x_n)$ in $L_p(\mathcal{M})$
\begin{align}\label{in2}
\left\| \sum_{n \in \mathbb{N}} \mathcal{E}_{n}(x_n) \right\|_{p} \leq C_{p} \left\| \sum_{n \in \mathbb{N}} x_n \right\|_{p}, ~ (1 \leq p < \infty) \qquad (DD_{p})
\end{align}
in which $C_{p}$ is a positive constant depending only on $p$.
\end{lemma}
By a duality argument, it is deduced from $DD_p$ (see \cite{J}) that
\begin{align}\label{in1}
\left\| \sup_n |\mathcal{E}_{n}(x)| \right\|_p \leq C_{p^\prime} \| x\|_p.
\end{align}
The next result is a non-commutative Stein inequality for the case when $1 \leq p < \infty$ and $q = \infty$, which is a consequence of $DD_p$. We state it for the sake of completeness.
\begin{theorem}[$S_{p, \infty}$ for $1 < p \leq \infty$] \label{th2}
Let $1 \leq p < \infty$ and $x = (x_n)_{n \in \mathbb{N}}$ be an arbitrary sequence in $L_p(\mathcal{M})$. Then
\begin{eqnarray*}
\left\| \left( \mathcal{E}_{n}(x_n) \right)_{n \in \mathbb{N}} \right\|_{L_p (\mathcal{M}, \ell_\infty)} \leq
C \left\| (x_n)_{n \in \mathbb{N}} \right \|_{L_p (\mathcal{M}, \ell_\infty)} \qquad (S_{p, \infty})
\end{eqnarray*}
for some positive constant $C$ depending only on $p$.
\end{theorem}
\begin{proof}
It follows from Theorem \ref{th1}(1) that any element in the unit ball of $L_p (\mathcal{M}, \ell_\infty)$ is a finite linear combination of positive elements in the same ball. Hence it is
enough to deduce the result for positive sequences. Using Theorem \ref{th1}(2) we have
\begin{align*}
\left\| \sup^+_{n} \mathcal{E}_{n}(x_n) \right\|_{p} & = \sup \left\{ \sum_{n \in \mathbb{N}} \tau \left( x_n \mathcal{E}_{n}(y_n) \right) ~ : y_n \in L_{p^\prime}^+ ~ \text{and} ~ \| \sum_{n \in \mathbb{N}} y_n \|_{p^\prime} \leq 1 \right\} \\
&=  C_{p^\prime} \sup \left\{ \sum_{n \in \mathbb{N}} \tau \left( x_n \frac{\mathcal{E}_{n}(y_n)}{C_{p^\prime}} \right) ~ : y_n \in L_{p^\prime}^+ ~ \text{and} ~ \| \sum_{n \in \mathbb{N}} y_n \|_{p^\prime} \leq 1 \right\} \\
&\leq  C_{p^\prime} \sup \left\{ \sum_{n \in \mathbb{N}} \tau \left( x_n z_n) \right) ~ : z_n \in L_{p^\prime}^+ ~ \text{and} ~ \| \sum_{n \in \mathbb{N}} z_n \|_{p^\prime} \leq 1 \right\}\\
& = C_{p^\prime} \left\| \sup^+_{n} x_n \right\|_{p} \quad \qquad (\text{by Theorem}\, \ref{th1}(2))
\end{align*}
To get the inequality above, we note that $\mathcal{E}_{n}$ is a positive map for any $n$ and use the inequality $DD_{p^\prime}$ \eqref{in2}.
\end{proof}
\begin{remark}
If we put $x_n = x$ for all $n$ in Theorem \ref{th2}, then $(S_{p, \infty})$ implies \eqref{in1}.
\end{remark}
\begin{corollary}
Let $1 \leq p < \infty$ and $x = (f_n)_{n \in \mathbb{N}}$ be an arbitrary sequence in $L_p(\Omega,\mathcal{F},\mathbb{P})$. Then
\begin{eqnarray*}
\left\| \sup_{n \in \mathbb{N}} |\mathbb{E}_{n}(f_n)| \right\|_p \leq
C \left\| \sup_{n \in \mathbb{N}} |f_n| \right \|_p
\end{eqnarray*}
for some positive constant $C$ depending only on $p$.
\end{corollary}

The next Proposition provides a non-commutative Stein inequality for the case when $1 \leq p=q < \infty$. We should notify that the case $p=q=\infty$ is proved in Theorem \ref{th2}.
\begin{proposition}[$S_{q, q}$ for $1 \leq q < \infty$] \label{msmz1}
Let $(x_n)_{n \in \mathbb{N}}$ be an arbitrary sequence in $L_q(\mathcal{M})$. Then
\begin{eqnarray*}
\left\| \left(\sum_{n \in \mathbb{N}} \left| \mathcal{E}_{n-1}(x_n) \right|^{q}\right)^{\frac{1}{q}} \right\|_q \leq
\left\| \left(\sum_{n \in \mathbb{N}} \left| x_n \right|^q \right)^{\frac{1}{q}}\right \|_q \qquad (S_{q,q})
\end{eqnarray*}
for every real number $q \geq 1$.
\end{proposition}
\begin{proof}
We show the desired inequality by some properties of trace and conditional expectation as follows.
\begin{eqnarray*}
&&\hspace{-1cm}\left\|\left(\sum_{n \in \mathbb{N}} \left| \mathcal{E}_{n-1}(x_n) \right|^{q}\right)^{\frac{1}{q}}\right \|_q^q \\
&=&
\tau \left( \sum_{n=1}^\infty \left| \mathcal{E}_{n-1}(x_n) \right|^{q} \right) =  \tau \left(\sup_{N \in \mathbb{N}}\left(\sum_{n=1}^N \left| \mathcal{E}_{n-1}(x_n) \right|^{q} \right) \right)\\
&=& \sup_{N \in \mathbb{N}} \sum_{n=1}^N \tau\left( \left| \mathcal{E}_{n-1}(x_n)\right|^{q} \right) \qquad \qquad \qquad \qquad \qquad \quad \quad (\text{by the normality of $\tau$})\\
&= & \sup_{N \in \mathbb{N}} \sum_{n=1}^N \left\| \mathcal{E}_{n-1}(x_n) \right\|_q^q \leq \sup_{N \in \mathbb{N}} \sum_{n=1}^N \| x_n \|_q^q \qquad \quad \qquad (\text{by the contractivity of $\mathcal{E}_n$}) \\
&= & \sup_{N \in \mathbb{N}} \tau \left(\sum_{n=1}^N \left| x_n \right|^{q}\right) = \tau \left( \sum_{n=1}^\infty \left| x_n \right|^{q} \right) \qquad \qquad (\text{again by the normality of $\tau$}) \\
&=& \left \| \left(\sum_{n \in \mathbb{N}} \left| x_n \right|^{q}\right)^{\frac{1}{q}}\right \|_q^q.
\end{eqnarray*}
\end{proof}


In \cite{PX}, the non-commutative Stein inequality $S_{p,2}$ is shown for $1 < p < \infty$ by applying Burkholder--Gundy inequality. Next, Qiu \cite{Qiu} obtained inequality $S_{1,2}$ for adapted sequences. The next corollary provides Stein inequality in the spaces $CR_p \left[ L_P(\mathcal{M}) \right]$.
\begin{corollary}
Let $1 < p < \infty$ and $x = (x_n)_{n \in \mathbb{N}}$ be an adapted sequence in $L_p(\mathcal{M})$. Then
\begin{eqnarray*}
\left\| \left( \mathcal{E}_{n-1} (x_n) \right)_n \right\|_{CR_p \left[ L_P(\mathcal{M}) \right]} \leq
C_p \left\| \left( x_n \right)_n \right\|_{CR_p \left[ L_P(\mathcal{M}) \right]},
\end{eqnarray*}
where $C_p$ is the constant appeared in $S_{p,2}$.
\end{corollary}
\begin{proof}
Suppose that $p \geq 2$. By replacing $x_n$ by $x_n^*$ in inequality $S_{p, 2}$ we get
\begin{align*}
\left\| \left( \mathcal{E}_{n-1} (x_n) \right)_n \right\|_{L_p \left(\mathcal{M}; \ell_2^R \right)} &=
\left\| \left(\sum_{n \in \mathbb{N}} |\mathcal{E}_{n-1} (x_n)^* |^{2}\right)^{\frac{1}{2}} \right\|_p = \left\| \left(\sum_{n \in \mathbb{N}} |\mathcal{E}_{n-1} (x_n^*) |^{2}\right)^{\frac{1}{2}} \right\|_p\\
& \qquad \qquad \quad (\text{by the $*$-preserving property of $\mathcal{E}_n$})\\
&\leq  C_p \left\| \left(\sum_{n \in \mathbb{N}} | x_n^* |^2 \right)^{\frac{1}{2}}\right \|_p \qquad \qquad \qquad \quad (\text{by $S_{p,2}$}) \\
&=  C_p \left\|  (x_n)_n  \right\|_{L_p \left(\mathcal{M}; \ell_2^R \right)}.
\end{align*}
Employing again inequality $S_{p,2}$, we infer that
\begin{align*}
\left\| \left( \mathcal{E}_{n-1} (x_n) \right)_n \right\|_{CR_p \left[ L_p(\mathcal{M}) \right]} \leq
C_p \left\| \left( x_n \right)_n \right\|_{CR_p \left[ L_P(\mathcal{M}) \right]}.
\end{align*}
Now let $1 < p < 2$. Assume that $x_n = a_n + b_n$ with $a_n$ and $b_n$ in $L_p \left(\mathcal{M}\right)$. Then $\mathcal{E}_{n-1} (x_n) = \mathcal{E}_{n} (a_n) + \mathcal{E}_{n} (b_n)$ is
a factorization of $\mathcal{E}_{n-1} (x_n)$ for all $n \in \mathbb{N}$. By applying inequality $S_{p,2}$ to $(a_n)_n$ and $(b_n)_n$, we have
\begin{align*}
\left\| \left( \mathcal{E}_{n-1} (a_n) \right)_n \right\|_{L_p \left(\mathcal{M}; \ell_2^C \right)} \leq
C_p \left\|  (a_n)_n  \right\|_{L_p \left(\mathcal{M}; \ell_2^C \right)}
\end{align*}
and
\begin{align*}
\left\| \left( \mathcal{E}_{n-1} (b_n) \right)_n \right\|_{L_p \left(\mathcal{M}; \ell_2^R \right)} \leq
C_p \left\|  (b_n)_n  \right\|_{L_p \left(\mathcal{M}; \ell_2^R \right)}.
\end{align*}
Hence
\begin{align*}
\left\| \left( \mathcal{E}_{n-1} (x_n) \right)_n \right\|_{CR_p \left[ L_P(\mathcal{M}) \right]} &=
\inf \left\{ \| (c_n)_n \|_{L_p \left(\mathcal{M}; \ell_2^C \right)} + \| (d_n)_n \|_{L_p \left(\mathcal{M}; \ell_2^R \right)} \right\} \\
&\leq  \left\| \left( \mathcal{E}_{n-1} (a_n) \right)_n \right\|_{L_p \left(\mathcal{M}; \ell_2^C \right)} + \left\| \left( \mathcal{E}_{n-1} (b_n) \right)_n \right\|_{L_p \left(\mathcal{M}; \ell_2^R \right)} \\
&\leq  C_p \left( \left\|  (a_n)_n  \right\|_{L_p \left(\mathcal{M}; \ell_2^C \right)} + \left\|  (b_n)_n  \right\|_{L_p \left(\mathcal{M}; \ell_2^R \right)} \right),
\end{align*}
where the infimum runs over all possible decompositions $\left( \mathcal{E}_{n-1} (x_n)\right)_n = (c_n)_n + (d_n)_n$ with $(c_n)_n \in L_p \left(\mathcal{M}; \ell_2^C \right)$ and $(d_n)_n \in L_p \left(\mathcal{M}; \ell_2^R \right)$.\\
Whence by taking infimum over all decompositions as above, we conclude the required inequality.
\end{proof}
\begin{remark}
Due to the dual version of Doob's inequality and noting that each element in the unit ball of $L_p (\mathcal{M}, \ell_1)$ is a sum of eight positive elements in the same ball
(see \ref{th1}(1)), we deduce that the linear map $\phi : L_p \left(\mathcal{M}; \ell_1 \right) \rightarrow L_p \left(\mathcal{M}; \ell_1 \right)$ with $\phi \left((x_n)\right)
= \left( \mathcal{E}_{n} (x_n) \right)$ is bounded. However, Junge obtained it with the bound $C_p$ as same as the constant, which is obtained from inequality $DD_p$ \eqref{in2}.
Moreover, $\phi$ is a bounded linear map on $L_p \left(\mathcal{M}; \ell_\infty \right)$ by Theorem \ref{th2}. Applying interpolation theorem one may deduce that
$\phi : L_p \left(\mathcal{M}; \ell_q \right) \rightarrow L_p \left(\mathcal{M}; \ell_q \right)$ is bounded. It seems that there is a problem to get the general Stein inequality. The problem is whether the equality
\begin{equation*}
\left\| (x_n) \right\|_{L_p \left(\mathcal{M}; \ell_q \right)} = \left\| \left( \sum_{n} x_n^q \right)^{\frac{1}{q}} \right\|_p
\end{equation*}
holds for any positive sequence $(x_n)$ or not. However we establish inequality $S_{p,q}$ for $1 \leq q \leq 2$ and $p \geq q$ in the following theorem.
\end{remark}
\begin{theorem}\label{th3}
Suppose that $1 \leq q \leq 2$, $p \geq q$, $(y_n)$ is a sequence of isometries in $\mathcal{M}$, and $(x_n)$ in $L_p^+(\mathcal{M})$ is an arbitrary positive sequence. Then
\begin{equation*}
\left\| \left(\sum_{n \in \mathbb{N}} \left(\mathcal{E}_{n} (y_n^*\,x_n\,y_n)\right)^{q}\right)^{\frac{1}{q}} \right\|_p \leq
C \left\| \left(\sum_{n \in \mathbb{N}} y_n^*\,x_n^q\,y_n \right)^{\frac{1}{q}}\right \|_p,
\end{equation*}
where $C$ is a positive constant depending only on $p$ and $q$.
\end{theorem}
\begin{proof}
We have
\begin{eqnarray*}
\left\| \left(\sum_{n \in \mathbb{N}} (\mathcal{E}_{n} (y_n^*\,x_n\,y_n) )^{q}\right)^{\frac{1}{q}} \right\|_p & \leq &
\left\| \left(\sum_{n \in \mathbb{N}} (\mathcal{E}_{n} \left((y_n^*\,x_n\,y_n)^q \right) \right)^{\frac{1}{q}} \right\|_p \\
&&(\text{by the Choi--Davis--Jensen inequality (see e.g. \cite{CHO,MOS})}\\
&&\text{and the operator monotonicity of $t^r$ ~ ($0 < r \leq 1$)} )\\
&\leq & \left\| \sum_{n \in \mathbb{N}} \mathcal{E}_{n} \left(y_n^*\,x_n^q\,y_n \right) \right\|_{\frac{p}{q}}^{\frac{q}{p^2}}\\
&& ~ (\text{by the Jensen operator inequality})\\
&\leq & C_{\frac{p}{q}} \left\| \sum_{n \in \mathbb{N}} y_n^*\,x_n^q\, y_n \right\|_{\frac{p}{q}}^{\frac{q}{p^2}} \quad (\text{by $DD_{\frac{p}{q}}$ \eqref{in2}})\\
&=& C_{\frac{p}{q}} \left\| \left(\sum_{n \in \mathbb{N}} y_n^*\,x_n^{q}\, y_n \right)^{\frac{1}{q}} \right\|_p.
\end{eqnarray*}
\end{proof}

\begin{corollary}[$S_{p,q}$ for $1 \leq q \leq 2$ and $q<p<\infty$]\label{msmz2}
Suppose that $1 \leq q \leq 2$, $q < p < \infty$ and $(x_n)$ in $L_p^+(\mathcal{M})$ is an arbitrary positive sequence. Then
\begin{equation}\label{in3}
\left\| \left(\sum_{n \in \mathbb{N}} (\mathcal{E}_{n} (x_n) )^{q}\right)^{\frac{1}{q}} \right\|_p \leq
C \left\| \left(\sum_{n \in \mathbb{N}} ( x_n )^q \right)^{\frac{1}{q}}\right \|_p,
\end{equation}
where $C$ is a positive constant depending only on $p$ and $q$.
\end{corollary}
\begin{proof}
Apply Theorem \ref{th3} with $y_n = 1$ for all $n \in \mathbb{N}$.
\end{proof}


The following results give some applications of $S_{p,q}$ \eqref{in3}, which is interesting on its own right.
\begin{corollary}[Semi-noncommutative case]
Let $(\Omega, \mathcal{F}, \mathbb{P})$ be a probability space, and $(\mathcal{F}_{n})_{n=0}^\infty$ be a filtration on it. If $1 \leq q \leq 2$, $q < p < \infty$, then there exists a positive constant $C$ such that for every positive stochastic process $(f_n)_{n=1}^\infty$ with values in $L_p(\mathcal{M})$
\begin{equation*}
\left\| \left(\sum_{n \in \mathbb{N}} (\mathcal{E}_{n} (f_n) )^{q}\right)^{\frac{1}{q}} \right\|_p \leq
C \left\| \left(\sum_{n \in \mathbb{N}} ( f_n )^q \right)^{\frac{1}{q}}\right \|_p,
\end{equation*}
in which $\mathcal{E}_{n} = \mathbb{E}_n \otimes id_{L_p(\mathfrak{M})}$, where $\mathbb{E}_n$ is the usual conditional expectation with respect to $\mathcal{F}_{n}$.
\end{corollary}
\begin{proof}
Set $\mathfrak{N} := L_{\infty}(\Omega, \mathcal{F}, \mathbb{P}) \otimes \mathfrak{M}$ and $\nu := \mathbb{E} \otimes \tau$, and apply \eqref{in3}.
\end{proof}

\begin{corollary}
If $1 \leq q \leq 2$, $q < p < \infty$, then there exists a positive constant $C$, depending only on $p$ and $q$, such that
\begin{equation*}
\left\| \left(\sum_{n \in \mathbb{N}} (\mathcal{E}_{n} (r_n) )^{q}\right)^{\frac{1}{q}} \right\|_p \leq C
\end{equation*}
for every sequence of mutually orthogonal projections $(r_n)_{n \geq 1}$ in $L_p(\mathcal{M})$ .
\end{corollary}

\begin{corollary}[Matrix algebra case]
If $1 \leq q \leq 2$, $q < p < \infty$, then there exists a positive constant $C$ such that for every sequence of positive matrices $(X_n)_n$ in $S_p$, the Schatten $p$-class on $l_2$, equipped with the usual trace it holds that
\begin{equation*}
\left\| \left(\sum_{n \in \mathbb{N}} (E_n\, X_n\, E_n) )^{q}\right)^{\frac{1}{q}} \right\|_p \leq
C \left\| \left(\sum_{n \in \mathbb{N}} ( X_n )^q \right)^{\frac{1}{q}}\right \|_p,
\end{equation*}
where $E_n$ is the operator projecting a sequence in $l_2$ into its $n$ first coordinates.
\end{corollary}
\begin{proof}
Regarding $B(l_2^n)$ as a subalgebra of $B(l_2)$, we have an increasing filtration $\left(B(l_2^n)\right)_{n\in \mathbb{N}}$. The corresponding conditional expectation from $B(l_2)$ onto $B(l_2^n)$ is $\mathcal{E}_n$ such that $\mathcal{E}_n(X) = E_nXE_n$ ~ $(X \in B(l_2))$, which extends to a contractive projection from $S_p$ onto $S_p^n$ (see \cite{Xu}). Now, apply \eqref{in3}.
\end{proof}

\end{document}